\documentclass[12pt]{amsart}
\usepackage{array,longtable}
\usepackage[utf8]{inputenc}
\usepackage{bbm}
\usepackage{amssymb,latexsym,fancyhdr,color,graphics}
\usepackage{amsmath, amsthm}
\usepackage{setspace}
\usepackage[dvips]{graphicx}
\usepackage[parfill]{parskip}

\newtheorem{lem}{Lemma}

\newtheorem{prop}{Proposition}
\newtheorem{thm}{Theorem}
\newtheorem{corr}{Corollary}
\newtheorem{ex}{Example}

\newtheorem{rem}{Remark}

\begin{document}

\title[On exponential functionals]{On exponential functionals of processes with independent increments}

\maketitle

\begin{center}
{\large P. Salminen, Department of Natural Sciences
Åbo Akademi University FIN-20500 Åbo Finland }\\
{\large and}\\
\hspace{-0.5cm}{\large  L. Vostrikova \footnote{$^1${This work is supported in part by  DEFIMATHS project of the Reseach Federation of ''Math\'{e}matiques de Pays de la Loire'' and PANORisk project of Pays de la Loire region.}} , LAREMA, D\'epartement de
Math\'ematiques, Universit\'e d'Angers, 2, Bd Lavoisier  49045,
\sc Angers Cedex 01}

\end{center} 
\vspace{0.2in}

\begin{abstract}
In this paper we study the exponential functionals of  the processes $X$ with independent increments  , namely
$$I_t= \int _0^t\exp(-X_s)ds, _,\,\, t\geq 0,$$ 
and also
$$I_{\infty}= \int _0^{\infty}\exp(-X_s)ds.$$
When $X$ is a semi-martingale with absolutely continuous characteristics, we derive recurrent  integral equations for  Mellin transform ${\bf E}( I_t^{\alpha})$, $\alpha\in\mathbb{R}$, of the integral functional $I_t$. Then we apply these  recurrent formulas to calculate the moments. We present also the corresponding results for the exponential functionals of  Levy processes, which hold under less restrictive conditions then in \cite{BY}. In particular, we obtain an explicit formula  for the moments of $I_t$ and $I_{\infty}$, and we precise  the exact number of finite moments  of $I_{\infty}$.

\end{abstract}

\begin{section}{Introduction}\label{s1}

\par The exponential functionals  arise in many areas : in the theory of self-similar Markov processes, in the theory of random processes in random environments, in the mathematical statistics, in the mathematical finance, in the insurance. In fact,  self-similar Markov processes are related with exponential functionals via Lamperti transform, namely self-similar Markov process can be written as an exponential of Levy process time changed by the inverse of exponential functional of the same Levy process (see \cite{L}). In the mathematical statistics the exponential functionals appear, for exemple, in the study of Pitman estimators (see \cite{NK}). In the mathematical finance the question is related to the perpetuities containing the liabilities, the perpetuities subjected to the influence of economical factors (see, for example, \cite{KR}), and also  with the prices of Asian options and related questions (see, for instance, \cite{CJY} and references therein). In the  insurance, this connection is made via the ruin problem, the problem in which the exponential functionals  appear very naturally (see, for exemple \cite{P}, \cite{A}, \cite{K} and references therein).
\par In  the case of Levy processes, the asymptotic behaviour of exponential functionals was studied in \cite{CPY}, in particular for $\alpha$-stable Levy processes. The authors also give  an integro-differential equation for the density of the law of exponential functionals, when this density w.r.t the Lebesgue measure exists. In \cite{BY}, for Levy subordinators, the authors give the formulas for the positive and negative moments, the Mellin transform and the Laplace transform. The questions related with the characterisation of the law of exponential functionals by the moments was also studied. General information about Levy processes  can be find in \cite{Sa}, \cite{B}, \cite{Ky}.
\par In more general setting, related to the Lévy case, the following functional
\begin{equation}\label{0} 
 \int _0^{\infty}\exp(-X_{s-})d\eta _s 
\end{equation}
 was studied by many authors, where $X=(X_t)_{t\geq 0}$ and $\eta= (\eta_t)_{t\geq 0}$ are independent Lévy processes.
 The interest to this functional can be explained by a very close relation between the distribution of this functional and the stationary density of the corresponding generalized Ornshtein-Uhlenbeck process. We recall that the generalized Ornshtein-Uhlenbeck process $Y=(Y_t)_{t\geq 0}$ verify the following differential equation
\begin{equation}\label{gen} 
 dY_t= Y_{t-}dX_t+d\eta_t
\end{equation} 
 When the jumps of the process $\eta$ are strictly bigger than -1, and $X_0=0$, the solution of this equation is
 \begin{equation}\label{gen1}
 Y_t= \mathcal{E}(X)_t \left(Y_0 + \int_0^t \frac{d\eta_s}{\mathcal{E}(X)_{s-}}\right)
 \end{equation}
 where $\mathcal{E}(X)$ is Doléan-Dade exponential (see \cite{JSh} for the details). Finally, if  we introduce another Lévy process $\hat{X}$ such that for all $t>~0$
 $$\mathcal{E}(X)_t = \exp(\hat{X}_t),$$  then  the integral of the type \eqref{0} appear in \eqref{gen1}. 
\par The conditions for finiteness of the integral \eqref{0} was obtained in \cite{EM}. The continuity properties of the law of this integral was studied in \cite{BLM}, where the authors give the condition for absence of the atoms and also the conditions for absolute continuity of the laws of integral functionals w.r.t. the Lebesgue measure. Under the assumptions which ensure the existence of the density of these functionals, the equations for the density  are given  in \cite{Be}, \cite{BeL}, \cite{KPS}. It should be noticed that taking the process $\eta$ being only drift, one can get $I_{\infty}$ as a special case of the Lévy framework.
\par In the papers \cite{PS} and \cite{PRV}, again for Levy process, the  properties of the exponential functionals $I_{\tau_q}$  was studied where $\tau_q$ is independent exponential random variable  of the parameter $q>0$. In the article \cite{PRV} the authors studied the existence of the density of the law of $I_{\tau_q}$, they give an integral equation for the density and the asymptotics of the law of $I_{\infty}$ at zero and at infinity, when $X$ is a positive subordinator.
The results given in \cite{PS}  involve analytic Wiener-Hopf factorisation, Bernstein functions and contain the conditions for regularity, semi-explicite expression and asymptotics for the distribution function of the exponential functional killed at the independent exponential time $\tau_q$. Despite numerous studies, the distribution properties of  $I_t$ and $I_{\infty}$ are known only in a limited number of cases. When $X$ is Brownian motion with drift, the distributions of  $I_t$ and $I_{\infty}$ was studied in \cite{D} and  for a big number of  specific processes $X$ and $\eta$, like Brownian motion with drift and compound Poisson process, the distributions of  $I_{\infty}$ was given in \cite{GP}. 
\par Exponential functionals for diffusions was studied in \cite{SW}. The authors considered exponential functionals stopped at first hitting time and  they derive  the Laplace transform of these functionals. To find the laws of such exponential functionals, the authors perform a numerical inversion of the corresponding Laplace transform. The relations between the hitting times and the occupation times for the exponential functionals was considered in \cite{SY}, where the versions of identities in law such as Dufresne's identity, Ciesielski-Taylor's identity, Biane's identity, LeGall's identity was considered.  
\par Howerever, the exponential functionals involving  non-homogeneous processes with independent increments (PII in short) have not been studied sufficiently up to now. Only a few results can be found in the literature. Some results  about the moments of  the exponential functional of this type are given in \cite{E}.  At the same time PII models for logarithme of the prices are quite natural in the mathematical finance, it is the case of the non-homogeneous Poisson process, the Levy process with deterministic time change, the integrals of Lévy processes with deterministic integrands, the hitting times for diffusions and so on (see for instance \cite{Sh}, \cite{CW}, \cite{BSh}).
\par The aim of this paper is  to study the exponential functionals of  the processes $X$ with independent increments, namely
\begin{equation}\label{deft}
I_t= \int _0^t\exp(-X_s)ds, \,\,\, t\geq 0,
\end{equation}
and also
$$I_{\infty}= \int _0^{\infty}\exp(-X_s)ds,$$
and give  such important characteristics of these exponential functionals as the moments and the  Laplace transforms and  the Mellin transforms.
\par For that we consider a real valued process $X=(X_t)_{t\geq 0}$ with independent increments  and $X_0=0$, which is a semi-martingale with respect to its natural filtration. We denote by $(B,C,\nu)$ a semi-martingale triplet of this process, which can be chosen deterministic (see \cite{JSh}, Ch. II, p.106). We suppose that
$B=(B_t)_{t\geq 0}$, $C=(C_t)_{t\geq 0}$ and $\nu$ are absolutely continuous with respect to the Lebesgue measure in $t$, i.e. that $X$ is an Ito process such that
\begin{equation}\label{abscont}
B_t= \int_0^t b_s\,ds, \,\,\,C_t= \int_0^t c_s\,ds, \,\,\, \nu(dt,dx)= dt K_t(dx)
\end{equation}
with the measurable functions $b=(b_s)_{s\geq 0}, c=(c_s)_{s\geq 0}$, and the kernel $K=(K_t(A))_{t\geq 0, A\in\mathcal{B}(\mathbb{R}\setminus\{0\})}$. For more information about the semi-martingales and the Ito processes see  \cite{JSh}.

\par We assume that the compensator of the measure of the jumps $\nu$ verify the usual relation: for each $t\in \mathbb{R}^+$
\begin{equation}\label{cond0}
\int_0^t\int_{\mathbb{R}\setminus\{0\}} (x^2\wedge 1)K_s(dx)\,ds < \infty.
\end{equation}
\par We recall that the characteristic function  of $X_t$  $$\phi_t(\lambda)={\bf E}\exp(i\lambda X_t)$$  is defined by the following expression: for $\lambda\in \mathbb{R}$
$$\phi_t(\lambda )= \exp\{i\lambda B_t - \frac{1}{2}\lambda ^2C_t + \int_0^t\int_{\mathbb{R}\setminus\{0\}}(e^{i\lambda x}-1-i\lambda x {\bf 1}_{\{|x|\leq 1 \}})\,K_s(dx)\,ds\}$$
which can be easily obtained by the Ito formula for semimartingales.
We recall also  that $X$ is a semi-martingale if and only if for all $\lambda\in \mathbb{R}$ the characteristic function  of $X_t$ is of finite variation in $t$ on finite intervals (cf. \cite{JSh}, Ch.2, Th. 4.14, p.106 ). Moreover, the process $X$ always can be written as a sum of a semi-martingale and a deterministic function which is not necessarily of finite variation on finite intervals.
\par From the formula for the characteristic function we can easily find the Laplace transform of $X_t$, if it exists:
$$ {\bf E} (e^{-\alpha X_t}) = e^{-\Phi (t, \alpha )}$$
Putting $\lambda = i\alpha$ in the previous formula, we get that
$$\Phi (t,\alpha) = \alpha B_t -\frac{1}{2}\alpha ^2C_t - \int_0^t\int_{\mathbb{R}\setminus\{0\}}(e^{-\alpha x}-1+\alpha x {\bf 1}_{\{|x|\leq 1 \}})\,K_s(dx)\,ds.$$
As known, in the case when $X$ is a Levy process with  parameters $(b_0,c_0, K_0)$, 
it holds
$${\bf E}(e^{-\alpha X_t})= e^{-t\Phi (\alpha)}$$
with
$$ \Phi (\alpha ) = \alpha b_0 -\frac{1}{2}\alpha ^2c_0 - \int_{\mathbb{R}\setminus\{0\}}(e^{-\alpha x}-1+\alpha x {\bf 1}_{\{|x|\leq 1 \}})\,K_0(dx).$$
\par In this article we derive a simple sufficient condition for finiteness of 
 ${\bf E} (I_t^{\alpha})$ with fixed $t>0$ (see Proposition 1). The main results of the paper, presented in Theorems 1 and 2, give recurrent integral equations for Mellin transform ${\bf E} (I_t^{\alpha})$. The starting point for the proof of these results is time reversal 
of the process $X$ at fixed time $t>0$, i.e. we  introduce a new process $Y^{(t)}=(Y^{(t)}_s)_{0\leq s\leq t}$ with $Y^{(t)}_s= X_t - X_{(t-s)-}$. In Lemma 1 we show that
$$ I_t= e^{-Y^{(t)}_t}\int_0^t e^{Y^{(t)}_s}ds.$$
Then, we prove that $Y^{(t)}$ is a process with independent increments and we identify its semi-martingale characteristics.  This time reversal plays an important role and  permits to replace  the process $(I_t)_{t\geq 0}$ which is not Markov, by a family of Markov processes $V^{(t)}=(V^{(t)}_s)_{0\leq s\leq t}$ indexed by $t>0$, where $V^{(t)}_s= e^{-Y^{(t)}_s}\int_0^s e^{Y^{(t)}_u}du$. For these family of Markov processes the general theory of stochastic processes can be applied to establish a recurrent integral equation of their Mellin transforms. Then, via Lemma 1, we can  reverse the time once more and get the integral recurrent equation for Mellin transform of  $I_t$. 

\par In Theorem \ref{t1} we consider the case of $\alpha \geq 0$ and we give a recurrent integral equation for the Mellin transform  of $I_t$.  In the Corollaries \ref{c1} and \ref{c2} we consider the case when $X$ is a Levy process. We give the formulas for positive moments  of $I_t$ and $I_{\infty}$ and also the formula for  the Laplace transform. The results for $I_{\infty}$   coincide, of course, with the ones given in  \cite{BY} for Levy subordinators. But it holds under less restrictive integrability conditions on $I_{\infty}$ and less restrictive  condition on Levy  measure at zero. We can also precise  the number of finite moments of $I_{\infty}$ (cf. Corollary \ref{c2} and Corollary \ref{c4}).
In Theorem \ref{t2} and  Corollaries \ref{c3} and \ref{c4} we present analogous result for the case $\alpha <0$.
\end{section}

\begin{section}{Finiteness of ${\bf E}(I_t^{\alpha})$ for fixed $t>0$.}

\par  In the following proposition we give  simple  sufficient conditions for the existence of the Mellin transform of $I_t$.

\begin{prop}\label{p2} Let $\alpha\in \mathbb{R}$ and $t>0$.  The condition
\begin{equation}\label{4}
\int_0^t\int_{|x|>1}e^{-\alpha x}\,K_s(dx) ds < \infty , 
\end{equation}
which is equivalent to  
\begin{equation}\label{5}
{\bf E}(e^{-\alpha X_t})<\infty ,
\end{equation}
 implies  that for  $\alpha \geq 1$ and $\alpha \leq 0$  
\begin{equation}\label{30aaa}
{\bf E}(I_t^{\alpha})<\infty .
\end{equation}
In addition, if \eqref{4} is valid for $\alpha = 1$, then for $0\leq \alpha\leq 1$ we have
\eqref{30aaa}.
\end{prop}
\vskip0.5cm
\begin{rem}\label{r1}\rm
Assume that X has only positive jumps, then for $\alpha >~0$ the condition \eqref{4} is always satisfied. In the same way, if $X$ has only negative jumps, then, of course, for  $\alpha <0$ the condition \eqref{4} is satisfied. If $X$ has bounded jumps,  the condition \eqref{4} is also satisfied.  In general case,   the condition \eqref{4}
is equivalent to the one of the following conditions: if $\alpha>0$, then
\begin{equation}\label{8} 
\int_0^t\int_{x<-1}e^{-\alpha x}K_s(dx) ds < \infty , 
\end{equation}
and if $\alpha<0$, then
\begin{equation}\label{99}
\int_0^t\int_{x>1}e^{-\alpha x}K_s(dx) ds < \infty . 
\end{equation}
In the case of Levy processes these conditions coincide with the ones given in \cite{Ky}, p. 79.
\end{rem}
\vskip0.5cm
\begin{rem}\label{r3}\rm It should be noticed that if the condition \eqref{4} is verified for
$\alpha '= \alpha +\delta$ with $\delta >0$ and  $\alpha >0$, then it is verified also for $\alpha $.
To see this,  apply H\"{o}lder inequality to the integrals in \eqref{8} and \eqref{99} with the parameters $p=\frac{\alpha +\delta}{\alpha}$ and $q=\frac{\alpha +\delta}{\delta}$. In fact, since $\frac{1}{p}+\frac{1}{q}=1$ and $\nu([0,t]\times ]-\infty,-1[)<\infty$ we have
for \eqref{8}:
$$\int_0^t\int_{x<-1}e^{-\alpha x}K_s(dx) ds \leq $$
$$\left(
\int_0^t\int_{x<-1}e^{-\alpha p x}K_s(dx) ds\right) ^{\frac{1}{p}}\,\left(\int_0^t\int_{x<-1}K_s(dx) ds\right) ^{\frac{1}{q}}
<\infty  $$
The same can be made for $\alpha<0$. In this case we take $\delta<0$.
\end{rem}

\begin{proof} We remark that for $t>0$, the law of $X_t$ coincide with the one-dimensional  law of Levy process $L$  with the triplet 
$$\left(\frac{1}{t}B_t,\, \frac{1}{t}C_t,\, \frac{1}{t}\int_0^t\int_{\mathbb{R}\setminus \{0\}}K_s(dx)ds\right)$$ at the time $t$. Then, the equivalence of \eqref{4} and \eqref{5} is a simple consequence of Kruglov theorem (cf. Theorem 25.3 in \cite{Sa}).
\par Let $\alpha\geq 1$ or $\alpha<0$ and $\tau$ be uniformly distributed on $[0,t]$ random variable independent of $X$ and . Then,
$${\bf E}(e^{-X_{\tau}})= {\bf E}(I_t/t)$$
 and applying the Jensen inequality we get
\begin{equation}\label{jensen}
(I_t/t)^{\alpha}\leq \frac{1}{t}\int_0^t e^{- \alpha X_s} ds
\end{equation}
Now, due to \eqref {5} and \eqref{abscont}, $f(s)={\bf E}(e^{-\alpha X_s})$, $s\in[0,t]$, is well-defined continuous function, and, hence, by Fubini theorem
$${\bf E}(I_t^{\alpha})\leq  t^{\alpha -1}\,\int _0^t\,{\bf E}(e^{-\alpha X_s}) ds<\infty .$$
\par If $\alpha =0$ the result is obvious, and if $0 < \alpha < 1$, then by the Hölder inequality we get that $${\bf E}(I_t^{\alpha})\leq ({\bf E}I_t)^{\alpha}<~\infty $$
and this ends the proof.
\end{proof}

Let us give some examples useful in mathematical finance.
\begin{ex}\rm
 Let $X$ be non-homogeneous Poisson process with intensity $(\lambda _s)_{s\geq 0}$. Then, $K_s(A) = \lambda _s\cdot \delta _{\{1\}}(A)$ for all $A\in \mathcal{B}(\mathbb{R}\setminus\{0\})$, where $\delta_{\{1\}}$ is delta-function at 1, and the condition \eqref{4} is satisfied as soon as $\int_0^t\lambda_s ds < \infty $. Hence, under the last condition, ${\bf E}(I_t^{\alpha})<\infty$ for all $\alpha\in\mathbb{R}$.
 \end{ex}
\begin{ex}\rm 
Let $L$ be Levy process with generating triplet $(b_0,c_0,K_0)$ and $X$ be the process $L$ time changed by deterministic continuously differentiable process $(\tau(t))_{t\geq 0}$, i.e. $X_t=  L_{\tau(t)}$.
Then, for $t\geq0$ 
$${\bf E}e^{-\alpha L_{\tau(t)}}= e^{-\Phi (\alpha)\tau(t)}$$
where $\Phi(\alpha)$ is Laplace exponent of $L$.
We write
$$\Phi(\alpha )\tau(t) = \alpha\,\tau(t)\,b_0 -\frac{1}{2}\alpha ^2\,\tau(t)\,c_0 - \tau(t)\,\int_{\mathbb{R}\setminus\{0\}}(e^{-\alpha x}-1+\alpha x\,{\bf 1}_{\{|x|\leq 1 \}} )\,K_0(dx)$$
and also, the Laplace exponent of $X$
$$\Phi(t, \alpha ) = \alpha B_t -\frac{1}{2}\alpha ^2C_t - \int_0^t\int_{\mathbb{R}\setminus\{0\}}(e^{-\alpha x}-1+\alpha x\,{\bf 1}_{\{|x|\leq 1 \}} )\,K_s(dx)\,ds$$
By the identification  we get that that 
 $B_t= b_0\tau(t)$, $C_t= c_0\tau(t)$, and for all $A\in\mathcal{B}(\mathbb{R}\setminus\{0\})$, $\nu(A, dt) = K_0(A) \tau '(t)dt$. Hence, $b_t=b_0\tau '(t)$, $c_t=c_0 \tau '(t)$, and all $A\in\mathcal{B}(\mathbb{R}\setminus\{0\})$, $K_t(A) = K_0(A) \tau '(t)$.
 The condition \eqref{4} is satisfied whenever
\begin{equation}\label{ex2}
\int_{|x|>1}e^{-\alpha x}K_0(dx) < \infty .
\end{equation}
Under this condition ${\bf E}(I_t^{\alpha})<\infty$.
\end{ex}
\begin{ex}\rm
Let $L$ be Levy process with generating triplet $(b_0,c_0,K_0)$ such that $K_0$ has a density $f_0$ w.r.t. a Lebesgue measure and 
$$X_t= \int_0^t g_sdL_s$$
where $g=(g_s)_{s\geq 0}$ is a not vanishing measurable square-integrable  function. To find the characteristics of $X$ we take a canonical decomposition of $L$, namely
$$L_t= b_0t+ \sqrt{c_0}W_t + \int _0^t\int_{\mathbb{R}\setminus \{0\}} x (\mu_X(ds, dx)-\nu(ds, dx))$$
where $W$ is standard Brownian motion.
We put this decomposition into the integral which defines $X$. We get by the identification that
$B_t= b_0\int_0^t g_sds$, $C_t= c_0\int_0^t g_s^2ds$. Moreover,
since $\Delta X_s= g_s \Delta L_s$ for $s>~0$,  we deduce by projection theorem that for any positive measurable function~$h$
$${\bf E}\int_0^t\int_{\mathbb{R}\setminus \{0\}}h(x) \mu_X(ds, dx)= {\bf E}\int_0^t\int_{\mathbb{R}\setminus \{0\}}h(x) K_s(dx) ds$$
and that the l.h.s. of the previous equality is equal to
$${\bf E}\int_0^t\int_{\mathbb{R}\setminus \{0\}}h(xg_s) \mu_L(ds, dx)= {\bf E}\int_0^t\int_{\mathbb{R}\setminus \{0\}}h(xg_s)K_0(dx)\,ds$$
Changing the variables we get  that $K_s(dx) = \frac{1}{|g_s|}\, f_0(\frac{x}{g_s})\,dx $ and that the condition
 \eqref{4} is satisfied whenever 
$$\int_0^t \int_{|x|>1}e^{-\alpha x g_s}\, K_0(dx)ds < \infty .$$  
Under the last condition,  ${\bf E}(I_t^{\alpha})<\infty$.
\end{ex}
\par In what follows we assume that $K$ verify the following stronger relation then \eqref{cond0}: for $t>0$
\begin{equation}\label{cond}
\int_0^t\int_{\mathbb{R}\setminus\{0\}} (x^2\wedge |x|) K_s(dx) ds < \infty .
\end{equation}
This supposition says, roughly speaking,  that the ''big'' jumps of the process $X$ are integrable, and it ensures that finite variation part of semi-martingale decomposition of $X$ remains deterministic. Moreover, the truncation of the jumps is no more necessary. In addition, Kruglov theorem can be applied to show that \eqref{cond} is equivalent to ${\bf E}(| X_t|)< \infty$ (cf. \cite{Sa}, Th. 25.3, p.159).
\end{section}
\begin{section}{Time reversal procedure}
\par We introduce, for fixed $t>0$,  a new process $Y^{(t)}=(Y^{(t)}_s)_{0\leq s\leq t}$ with $Y^{(t)}_s= X_t-X_{(t-s)-}$. To simplify the notations we anyway omit the index $(t)$ and write $Y_s$ instead of $Y_s^{(t)}$.

First of all we establish the relation between $I_t$ and the process $Y$.
\begin{lem}\label{ll1}
For $t>0$ the following equality holds:
$$ I_t= e^{-Y_t}\int _0^t e^{Y_s}ds $$ 
\end{lem}
\begin{proof} Using the definition of the process $Y$ and the assumption that $X_0=0$ we have  
$$e^{-Y_t}\int _0^t e^{Y_s}ds =  \int _0^t e^{-Y_t +Y_s}ds=
\int _0^t e^{-X_{(t-s)-}+X_0}ds $$
$$ \hspace{1.8cm}=\int_0^t e^{-X_{t-s}}ds= \int_0^t e^{-X_s}ds = I_t,$$
since the integration of the both versions of the process w.r.t. the Lebesgue measure gives  the same result. 
\end{proof}

We will show that the process $Y$ is PII and we will give its semi-martingale triplet with respect to its natural filtration.  For that we put 
$$\bar{b}_u=\left\{ \begin{array}{ll}
b_{t-u}&\mbox{for}\, 0\leq u< t,\\
b_t &\mbox{for}\,u=t,
\end{array}\right.$$
which can be written also as
$$\bar{b}_u={\bf 1}_{\{t\}}(u)(b_t-b_0)+ b_{t-u},$$
where ${\bf 1}_{\{t\}}$ is indicator function of the set $\{t\}$.
We do the similar definitions for $\bar{c}$ and ~$\bar{K}$:
$$\bar{c}_u= {\bf 1}_{\{t\}}(u)(c_t-c_0)+c_{t-u},$$
$$\bar{K}_u(A)= {\bf 1}_{\{t\}}(u)(K_t(A)-K_0(A))+ K_{t-u}(A)$$
 for all $A\in\mathcal{B}(\mathbb{R}\setminus\{0\})$. 
\begin{lem}\label{ll2}
The process $Y$ is a process with independent increments, it is a semi-martingale with respect to its natural filtration, and its semi-martingale triplet  $(\bar{B}, \bar{C}, \bar{\nu})$ is given  by :
\begin{equation}\label{cat}
\bar{B_s}= \int_0^s \bar{b}_u du, \,\,\bar{C_s}= \int_0^s \bar{c}_u du,\,\,
\bar{\nu}(ds,dx) = \bar{K}_s(dx) \,ds \,,
\end{equation}
where $0\leq s \leq t$.
\end{lem}
\it Proof \rm Let us take $0=s_0<s_1<s_2<\cdots <s_n= t$ with $n\geq 2$. Then the increments $(Y_{s_k}-Y_{s_{k-1}})_{1\leq k\leq n}$ of the process $Y$ are equal to
$(X_{(t-s_{k-1})-}-X_{(t-s_{k})-})_{1\leq k\leq n}$ and $0=t-s_n<t-s_{n-1}<\cdots <t-s_0=~t$. From the fact that $X$ is the process with independent increments, the characteristic function of the vector $(X_{t-s_{k-1}-h}-X_{t-s_{k}-h})_{1\leq i\leq n}$ with small $h>0$, can be written as a product of the corresponding characteristic functions. Namely, for any real constants $(\lambda_k)_{0\leq k\leq n}$ we get:
$${\bf E}\exp(i\sum_{k=1}^n \lambda_k (X_{t-s_{k-1}-h} - X_{t-s_k-h})) = \prod_{k=1}^n {\bf E}\exp(i\lambda _k (X_{t-s_{k-1}-h} - X_{t-s_k-h}))$$
Then, passing to the limit as $h\rightarrow 0+$,
$${\bf E}\exp(i\sum_{k=1}^n\lambda_k (X_{(t-s_{k-1})-} - X_{(t-s_k)-})) = \prod_{k=1}^n {\bf E}\exp(i\lambda _k(X_{(t-s_{k-1})-} - X_{(t-s_k)-}))$$
Hence, $Y$ is a process with independent increments.
\par We know that we can identify the semi-martingale characteristics w.r.t. the natural filtration of the process with independent increments    from the characteristic function of this process. We notice that $Y_s=X_t-X_{(t-s)-}$ and by the independence of $Y_s$ and $X_{(t-s)-}$
$${\bf E}(e^{i\lambda X_t}) = {\bf E}(e^{i\lambda Y_s}){\bf E}(e^{i\lambda X_{(t-s)-}})$$
Then,
$${\bf E}\exp (i\lambda Y_s)={\bf E}\exp (i\lambda X_t)/{\bf E}\exp (i\lambda X_{(t-s)-})=$$ $$ \exp\{i\lambda\int_{t-s}^tb_u du - \frac{1}{2}\lambda ^2\int_{t-s}^tc_u du + \int_{t-s}^t\int_{\mathbb{R}\setminus\{0\}}(e^{i\lambda x}-1-i\lambda x )\,K_u(dx)\,du\}$$
We substitute $u$ by $u'=t-u$ in the integrals to obtain
$$\hspace{-3cm}{\bf E}\exp (i\lambda Y_s)= $$ $$ \exp\{i\lambda\int_0^s\overline{b}_udu - \frac{1}{2}\lambda ^2\int_0^s\overline{c}_u du + \int_0^s\int_{\mathbb{R}\setminus\{0\}}(e^{i\lambda x}-1-i\lambda x )\,\overline{K}_u(dx)\,du\}$$
Therefore, the characteristics of $Y$ are as in  \eqref{cat}
and the proof is complete.$\Box$
\end{section}
\begin{section}{Recurrent formulas for the Mellin transform  of $I_t$  with $\alpha \geq 0$.}
Let us consider two important processes  related with the process $Y$, namely the process $V=(V_s)_{0\leq s\leq t}$ and $J=(J_s)_{0\leq s\leq t}$ defined via
$$ V_s= e^{-Y_s}J_s,\hspace{1cm} J_s= \int_0^s e^{Y_u}du.$$
We underline that the both  processes depend of the parameter $t$.
\par We remark that according to  Lemma \ref{ll1}
, $I_t=V_t$  for each $t\geq 0$.
For $\alpha \geq 0$ and $t\geq 0$  we introduce  the Mellin transform of $I_t$ of  the parameter ~$\alpha$:
$$m^{(\alpha )}_t= {\bf E}(I^{\alpha}_t)= {\bf E}(e^{-\alpha Y_t}J^{\alpha}_t)$$
and the Mellin transform for shifted process:
$$m_{s,t}^{(\alpha )}= {\bf E}\left[ \left(\int_{s}^{t}e^{-(X_u-X_{s-})}du\right)^{\alpha}\right]$$
Notice that $m_{0,t}^{(\alpha )}=m_{t}^{(\alpha )}$. Notice also that
$$m_{s,t}^{(\alpha )}= {\bf E}\left[ \left(\int_{s}^{t}e^{-(X_u-X_{s})}du\right)^{\alpha}\right]$$
In fact, $X_u-X_s= X_u-X_{s-}-\Delta X_s$ and
$$\int_{s}^{t}e^{-(X_u-X_{s-})}du=e^{-\Delta X_s}\int_{s}^{t}e^{-(X_u-X_{s})}du$$
Since $\Delta X_s$ and $(X_u-X_s)_{u\geq s}$ are independent, and ${\bf E}(e^{-\alpha \Delta X_s})=1$, we get the equality of two expressions for $m_{s,t}^{(\alpha )}$.
\par We introduce also two functions: for $0\leq s\leq t$
\begin{equation}\label{H}
H^{(\alpha)}_s= \alpha b_s  - \frac{1}{2}\alpha ^2 c_s - \int_{\mathbb{R}\setminus\{0\}} (e^{-\alpha x}-1+\alpha x) K_s(dx) 
\end{equation}
and
\begin{equation}\label{bH}
\bar{H}^{(\alpha )}_s= {\bf 1}_{\{t\}}(s)(H_t^{(\alpha)}- H_0^{(\alpha )})+ H^{(\alpha)}_{t-s}
\end{equation}
These functions represent the derivatives w.r.t. $s$, of the Laplace exponents $\Phi(s,\alpha)$ and $\bar{\Phi}(s, \alpha)$. We recall that
$$\Phi (s,\alpha) = \alpha B_s -\frac{{\alpha}^2}{2}C_s - \int_0^s\int_{\mathbb{R}\setminus\{0\}}(e^{-\alpha x}-1+\alpha x )\,\nu (du, dx)$$
and
$$\bar{\Phi} (s,\alpha) = \alpha \bar{B}_s -\frac{{\alpha}^2}{2}\bar{C}_s -\int_0^s \int_{\mathbb{R}\setminus\{0\}}(e^{-\alpha x}-1+\alpha x )\bar{\nu} (du, dx)$$
where $\nu$ and $\bar{\nu}$ are the compensators of the jump measure of $X$ and $Y$ respectively.
We notice, that these functions are well-defined  under condition \eqref{4}. We also notice that 
$$\bar{\Phi} (s,\alpha) =  \Phi (t,\alpha) -\Phi (t-s,\alpha) $$
\par Our aim now is to obtain a recurrent integral equation for the Mellin transform of $I_t$. For condition \eqref{rt1} below see Remarks \ref{r1} and \ref{r3}.
\begin{thm}\label{t1} Let $\alpha\geq 1$ be fixed and assume that $t>0$. Suppose that \eqref{abscont} and \eqref{cond} hold and  there exists $\delta >0$ such that
\begin{equation}\label{rt1}
\int_0^t\int_{x<-1} e^{-(\alpha +\delta ) x} K_s(dx)\,ds < \infty .
\end{equation}
Then,  $m^{(\alpha)}_t$ is well-defined and 
the following recurrent integral equation holds
\begin{equation}\label{8b}
m^{(\alpha)}_t = \alpha \int _0^t m^{(\alpha -1)}_{u,t}
 \,\,e^{-\Phi (u,\alpha)}\, du
\end{equation}
 If $X$ is Levy process, then  for all $t>0$ 
\begin{equation}\label{8aab}
m^{(\alpha)}_t = \alpha e^{-\Phi (\alpha)\,t}\int _0^t m^{(\alpha -1)}_s\,e^{\Phi ( \alpha)\,s}ds
\end{equation}
Moreover, 
\begin{equation}\label{8aa}
\frac{d}{dt}\left[m^{(\alpha)}_t\right] = - m^{(\alpha)}_t \Phi(\alpha )  + \alpha m^{(\alpha -1)}_t
\end{equation} 
\end{thm}
\begin{proof} From Lemma \ref{ll2} we know that $Y=(Y_s)_{0\leq s\leq t}$ is a process with independent increments which is a semi-martingale, i.e.
$Y_s= \bar{B}_s +\bar{M}_s$, where $\bar{B}$ is a deterministic process of finite variation on finite intervals and $\bar{M}$ is a local martingale. But the local martingales with independent increments are always the martingales (see \cite{ShCh}).

For $n\geq 1$ we introduce the stopping times 
$$\tau_n= \inf \{0\leq  s\leq t: V_s\geq n\,\mbox{or}\, \exp(-Y_s)\geq n\}$$
with $\inf\{\emptyset\}=+\infty$. For fixed $s$,  $0< s<t$, we write the Ito formula for  $V^{\alpha}_{s\wedge \tau_n}$ :
\begin{equation}\label{v1}
V^{\alpha}_{s\wedge \tau_n} = \alpha \int_0^{s\wedge \tau_n} V_{u-}^{\alpha -1} dV_u + \frac{1}{2}\alpha (\alpha -1)\int_0^{s\wedge \tau_n} V_{u-}^{\alpha -2}d<V^c>_u$$
$$ + \int_0^{s\wedge \tau_n}\int_{\mathbb{R}\setminus\{0\}}\left( (V_{u-}+x)^{\alpha} - V^{\alpha}_{u-} - \alpha V^{\alpha -1}_{u-} x\right) \mu_V(du,dx)
\end{equation}
where $\mu_V$ is the measure of the jumps of $V$.
Using integration by part formula, we have:
\begin{equation}\label{v2}
dV_u= du + J_u d( e^{-Y_u})
\end{equation}
Now, again by the Ito formula, we get
\begin{equation}\label{exp}
e^{-Y_u}=e^{-Y_0}-\int_0^ue^{-Y_{v-}}dY_v + \frac{1}{2}\int_0^ue^{-Y_{v-}} d<Y^c>_v
\end{equation}
$$\hspace{4cm} + \int_0^u\int_{\mathbb{R}\setminus\{0\}}e^{-Y_{v-}}(e^{-x}-1+x)\mu _Y(dv,dx)$$ 
Then, putting \eqref{exp} into  \eqref{v2}, we obtain
\begin{equation}\label{v3} dV^c_u = - e^{-Y_{u-}}J_{u}\, dY^c_u=-V_{u-}\,dY^c_u,
\end{equation}
$$d<V^c>_u \,= \,V_{u-}^2 \,d<Y^c>_u$$
and 
\begin{equation}\label{v4}
 \Delta V_u= e^ {-Y_{u-}}J_{u}(e^{-\Delta Y_u}-1)= V_{u-}(e^{-\Delta Y_u}-1),
\end{equation}
where $\Delta V_u= V_u-V_{u-}$ and $\Delta Y_u= Y_u-Y_{u-}$.
The previous relations imply that
\begin{equation}\label{9}
V^{\alpha}_{s\wedge \tau_n} = \alpha \int_0^{s\wedge \tau_n} V_{u-}^{\alpha -1} du 
\end{equation} 
$$+ \alpha \int_0^{s\wedge \tau_n}J_{u}\, V_{u-}^{\alpha -1} d(e^{-Y_u}) + \frac{1}{2}\alpha (\alpha -1)\int_0^{s\wedge \tau_n} V_{u-}^{\alpha }d<Y^c>_u$$
$$\hspace{3cm} + \int_0^{s\wedge \tau_n}\int_{\mathbb{R}\setminus\{0\}}V_{u-}^{\alpha}\left( e^{-\alpha x} -1-\alpha(e^{-x}-1)\right)\mu_Y(du,dx)$$
Now, to use in efficient way the Ito formula for $e^{-Y_u}$ given before, we introduce
the processes $A=(A_u)_{0\leq u\leq t}$ and $N=(N_u)_{0\leq u\leq t}$ via
$$A_u= \int_0^u e^{-Y_{v-}}[-d\bar{B}_v+\frac{1}{2}d\bar{C}_v]+\int_0^u \int_{\mathbb{R}\setminus\{0\}} e^{-Y_{v-}}(e^{-x}-1+x)\bar{\nu}(dv, dx)$$
$$N_u=- \int_0^u e^{-Y_{v-}}d\bar{M}_v+\int_0^u \int_{\mathbb{R}\setminus\{0\}}e^{-Y_{v-}}(e^{-x}-1+x)[\mu _Y(dv,dx)-\bar{\nu}(dv, dx)]$$ 
We notice that  $A$ is a process of locally bounded variation and $N$ is a local martingale with localizing sequence $(\tau _n)_{n\geq 1}$, since $\bar{B}$, $\bar{C}$
are of bounded variation on bounded intervals and $$\int_0^u\int_{\mathbb{R}\setminus\{0\}}(e^{-x} -1+x)\bar{K}_s(dx)  ds < \infty .$$
\par From  \eqref{exp}  we get that
\begin{equation}\label{v5}
e^{-Y_u} =  e^{-Y_0} +A_u + N_u.
\end{equation}
We incorporate this semi-martingale decomposition into \eqref{9} and we consider its martingale part.
This martingale part is represented by the term
$$\alpha\int_0^{s\wedge\tau_n}V_{u-}^{\alpha -1}\, J_u dN_u=$$
$$ \alpha\int_0^{s\wedge\tau_n}V_{u-}^{\alpha }\,[- d\bar{M}_u + \int_{\mathbb{R}\setminus\{0\}}  (e^{-x}-1+x)(\mu _Y(dv,dx)-\bar{\nu}(dv, dx))]$$
which is a local martingale. Let $(\tau'_n)_{n\geq 0}$ be a localizing sequence for this local martingale and let $\bar{\tau}_n= \tau_n\wedge\tau'_n$. Then we do additional stopping with $\tau'_n$ in previous expressions,
and we take mathematical expectation. Using the fact that the expectations of martingales starting from zero are equal to zero and also applying the projection theorem, we obtain:
\begin{equation}\label{rel}
{\bf E}(V^{\alpha}_{s\wedge \bar{\tau}_n}) = \alpha {\bf E}\left(\int_0^{s\wedge \bar{\tau}_n} V_{u-}^{\alpha -1} du\right)
\end{equation}
$$ + \alpha {\bf E}\left(\int_0^{s\wedge \bar{\tau}_n} V_{u-}^{\alpha -1} J_u \,dA_u\right)+ 
\frac{1}{2}\alpha (\alpha -1){\bf E}\left(\int_0^{s\wedge\bar{\tau}_n} V_{u-}^{\alpha }\,d\bar{C}_u\right)$$
$$ +{\bf E}\left( \int_0^{s\wedge \bar{\tau}_n}\int_{\mathbb{R}\setminus\{0\}}V_{u-}^{\alpha}\left[ e^{-\alpha x} -1-\alpha(e^{-x}-1)\right]\bar{\nu}(du, dx)\right)$$
and, hence,
\begin{equation}\label{rel1}
{\bf E}(V^{\alpha}_{s\wedge \bar{\tau}_n}) = \alpha {\bf E}\left(\int_0^{s\wedge \bar{\tau}_n} V_{u}^{\alpha -1} du\right) - {\bf E} \left(\int_0^{s\wedge \tau_n}V_{u-}^{\alpha}\,\,d\bar{\Phi}(u, \alpha) \right)
\end{equation}
\par We remark that $\bar{\tau}_n\rightarrow +\infty\, ({\bf P}-a.s.)$ as $n\rightarrow +\infty$. To pass to the limit as $n\rightarrow \infty$ in r.h.s. of the above equality, we use  the Lebesgue monotone convergence theorem for the first term and the Lebesgue dominated convergence theorem for the second term. In fact,   for the second term we have using \eqref{H} and \eqref{bH} :
$$ \left|\int_0^{s\wedge \bar{\tau}_n}V_{u-}^{\alpha}\,\,d\bar{\Phi}(u, \alpha)\right|\leq \int_0^{t}V^{\alpha}_{u-}\,|\bar{H}_u^{(\alpha )}|du$$
In addition,
$${ \bf E}\left(\int_0^{t}V^{\alpha}_{u-}\,|\bar{H}_u^{(\alpha )}|du\right)\leq 
\sup_{0\leq u \leq t}{\bf E}(V^{\alpha}_u)\,\int_0^{t}|\bar{H}_u^{(\alpha )}|du$$
The function $(\bar{H}_u^{(\alpha )})_ {0\leq u\leq t}$ is deterministic function, integrable on finite intervals. Hence, it remains to show that
\begin{equation}\label{sup0}
\sup_{0\leq s \leq t}{\bf E}(V^{\alpha}_s) < \infty .
\end{equation} 
By the Jensen inequality similar to \eqref{jensen} we obtain :
$$V^{\alpha}_s\leq s^{\alpha -1}\int_0^s e^{\alpha (Y_u - Y_s)}du$$
 and due to the fact that $Y_u-Y_s= X_{(t-s)-}- X_{(t-u)-}$, we get
\begin{equation}\label{sup1}
{\bf E}(V^{\alpha}_s) \leq s^{\alpha -1}\int_0^s {\bf E}( e^{\alpha ( X_{t-s}- X_{t-u})})\,du
\end{equation}
Since the process $X$ is a process with independent increments, we have for $0\leq ~u~\leq ~s\leq t$
$$ {\bf E}( e^{\alpha ( X_{t-s}- X_{t-u})})=  {\bf E}( e^{-\alpha  X_{t-u}})/ {\bf E} (e^{-\alpha  X_{t-s}})= $$
$$\exp\{ -\int_0^{t-u} H^{(\alpha)}_rdr +\int_0^{t-s} H^{(\alpha)}_rdr\}\leq \exp\{ \int_{0}^{t}| H^{(\alpha)}_r|dr \}$$
Due to the Remark \ref{r1} and the condition \eqref{rt1}, $H^{(\alpha )}$ is integrable function on finite intervals,  and hence,
$$\sup_{0\leq s \leq t}{\bf E}(V^{\alpha}_s)\leq t^{\alpha}\exp\{ \int_{0}^{t}| H^{(\alpha)}_r|dr \} <\infty$$
\par To pass to the limit in the l.h.s. of \eqref{rel1}, we show that  the family of $(V^{\alpha}_{s\wedge\tau_n})_{n\geq 1}$ is uniformly integrable, uniformly in $0\leq s\leq t$.
For that we recall that $Y_s= \bar{B}_s + \bar{M}_s$ where $\bar{B}$ is a drift part and $\bar{M}$ is a martingale part of $Y$, and we introduce
$$\bar{V}_s = e^{-\bar{M}_s}\int _0^s e^{\bar{M}_u} du$$
Then,
$$V_s= e^{-Y_s}\int _0^s e^{Y_u} du= e^{-\bar{M}_s-\bar{B}_s}\int _0^s e^{\bar{M}_u +\bar{B}_u} du$$
$$\leq \sup_{0\leq u\leq s} ( e^{\bar{B}_u - \bar{B}_s})\, \bar{V}_s\leq e^{Var(\bar{B})_t} \,\, \bar{V}_s$$
where $Var(\bar{B})_t$ is the variation of $\bar{B}$ on the interval $[0,t]$. This quantity is deterministic and bounded, hence, it is sufficient to prove uniform integrability of $(\bar{V}^{\alpha}_{s\wedge\tau_n})_{n\geq 1}$.
\par Next, we show that $(\bar{V}_s)_{0\leq s\leq t}$ is a submartingale w.r.t. a natural filtration of $Y$. Let $s'>s$, then
$$\hspace{-4cm}{\bf E}(\bar{V}_{s'}\,|\, \mathcal{F}_s)= {\bf E}(e^{-\bar{M}_{s'}}\int _0^{s'} e^{\bar{M}_u} du\,|\,  \mathcal{F}_s)=$$
 $${\bf E}(e^{-(\bar{M}_{s'}-\bar{M}_{s})}\,[\bar{V}_s + e^{-\bar{M}_s}\int _s^{s'} e^{\bar{M}_u} du]\,|\,  \mathcal{F}_s)\geq$$
$$\hspace{4cm}{\bf E}(e^{-(\bar{M}_{s'}-\bar{M}_{s})} \bar{V}_s\,|\,  \mathcal{F}_s) = \bar{V}_s\, {\bf E}(e^{-(\bar{M}_{s'}-\bar{M}_{s})})$$
The expression for ${\bf E}(e^{-(\bar{M}_{s'}-\bar{M}_{s})})$ can be find from the expression of the characteristic exponent of $Y$ without its drift part: 
$$\hspace{-8cm}{\bf E}(e^{-(\bar{M}_{s'}-\bar{M}_{s})}= $$
$$\hspace{2cm}\exp \left(\frac{1}{2}\int_s^{s'}c_udu +                   \int_s^{s'}\int _{\mathbb{R}\setminus\{0\}}(e^{-x}-1+x)\bar{\nu}(du,dx) \right)\geq 1.$$
Then, $(\bar{V}_s^{\alpha})_{0\leq s \leq t}$ is  a submartingale, and by Doob stopping theorem ($P$-a.s.)
$${\bf E}( \bar{V}^{\alpha}_s\,|\, \mathcal{F}_{s\wedge\tau_n})\geq 
\bar{V}^{\alpha}_{s\wedge\tau _n}.$$  Hence, for all $n\geq 1$, $c>0$ and ${\bf I}(\cdot)$ indicator function
$$\hspace{-0.5cm}{\bf E}( \bar{V}^{\alpha}_{s\wedge\tau _n} {\bf I}_{\{\bar{V}_{s\wedge\tau _n}>c\}}) \leq 
{\bf E}({\bf E}(\bar{V}_s^{\alpha}\,|\, \mathcal{F}_{s\wedge\tau _n}) {\bf I}_{\{{\bf E}(\bar{V}_s^{\alpha}\,|\, \mathcal{F}_{s\wedge\tau _n})>c\}})\leq$$
$$\hspace{5cm} c^{-\frac{\delta}{\alpha}}{\bf E}({\bf E}(\bar{V}_s^{\alpha}\,|\, \mathcal{F}_{s\wedge\tau_n})^{\frac{\alpha +\delta}{\alpha}})\leq c^{-\frac{\delta}{\alpha}}{\bf E}(\bar{V}_s^{\alpha +\delta})$$
Finally, we show in the same way as the proof of  \eqref{sup0} that 
$$\sup_{0\leq s\leq t}{\bf E}(\bar{V}_s^{\alpha +\delta})<\infty $$
and it proves the uniform integrability of the family  $(V^{\alpha}_{s\wedge\tau_n})_{n\geq 1}$  uniformly in $0\leq s\leq t$.
 \par After limit passage, we get that
\begin{equation}\label{rt2}
{\bf E}(V_s^{\alpha}) = - \int_0^{s}{\bf E}(V_u^{\alpha}) \,d\bar{\Phi}(u,\alpha ) + \alpha\int_0^s
{\bf E}(V_u^{\alpha -1})du
\end{equation}
We see that each term of this equation is differentiable w.r.t. $s$ for $s<t$. In fact, the family $(V^{\alpha}_s)_{0\leq s\leq t}$ is uniformly integrable and  the function $({\bf E}(V_s^{\alpha}))_{0\leq s \leq t}$ is continuous in $s$ as well as $(\bar{\Phi}(s,\alpha))_{0\leq s <t}$. We calculate the derivatives in $s$ of both sides of the above equation  and
we solve the corresponding linear equation to obtain:
$${\bf E}(V_s^{\alpha})= \alpha e^{-\bar{\Phi}(s,\alpha )}\int_0^s
{\bf E}(V_u^{\alpha -1})\,e^{\bar{\Phi}(u,\alpha )}\,du$$
Now, we write that $\bar{\Phi}(u,\alpha )-\bar{\Phi}(s,\alpha )=-\Phi(t-u,\alpha )+ \Phi (t-s,\alpha)$
and we let $s\rightarrow t-$ to get
$${\bf E}(V_t^{\alpha})=\alpha \int_0^t
{\bf E}(V_u^{\alpha -1})\,e^{-\Phi(t-u,\alpha )}\,du$$
Notice that $I_t=V_t$ and, hence, ${\bf E}(V_t^{\alpha})=m^{(\alpha)}_{t}$.
Since
$$V_u= \int_0^ue^{Y_v-Y_u}dv = \int_0^u e^{-(X_{(t-v)-}-X_{(t-u)-})}dv= \int_{t-u}^t e^{-(X_v-X_{(t-u)-})}dv
$$
we also have
 ${\bf E}(V_u^{\alpha -1})=m^{(\alpha -1)}_{t-u,t}$.
Then, we obtain \eqref{8b} after the change of  variables in integrals replacing $t-u$ by $u$.
\par  In particular case of Levy process, we use \eqref{8b} and we take into account that $ m^{(\alpha -1)}_{u,t}= m^{(\alpha -1)}_{t-u}$ and that $\Phi(u,\alpha)=\Phi (\alpha)\,u$. Then, after the change of  variables, the formula \eqref{8aab} follows as well as \eqref{8aa}.
\end{proof}
\begin{rem}\rm It should be noticed that the relation \eqref{8b} can be obtained with another technique (see \cite{SaWo}), based on the approach of \cite{BY}.
\end{rem}
\par Now we will apply our results to calculate the moments.
\begin{ex}\rm  We consider Levy process $L=(L_t)_{t\geq 0}$ with generating triplet $(b_0,c_0,K_0)$ starting from 0, and time changed by deterministic process $(\tau_t)_{t\geq 0}$ with $\tau_t= r\ln (1+t)$ with $r>0$, i.e. $X_t= L_{\tau_t}$ for $t\geq 0$. Then, by change of variables $u=r\ln(1+s)$ we get
$$I_t= \int^t_0 e^{-L_{r\ln (1+s)}}ds=\frac{1}{r} \int_0^{r\ln(1+t)} e^{-(L_u -u/r)}du=
\frac{1}{r} \int_0^{r\ln(1+t)} e^{-\tilde{L}_u }du$$
where $\tilde{L}$ is Levy process with generating triplet $(b_0- \frac{1}{r},c_0,K_0)$.
We denote by $\tilde{\Phi}$  the Laplace exponent of $\tilde{L}$. Then for $k\geq 0$,
$\tilde{\Phi}(k) = \Phi (k) -\frac{k}{r}$ and
$${\bf E}(I_t)= \frac{1}{r}\int_0^{r\ln(1+t)} e^{-\tilde{\Phi}(1)u} du=
\left\{\begin{array}{ll}
\frac{1-(1+t)^{-r\Phi(1) +1}}{r\Phi(1)-1}&\mbox{if}\, r\Phi(1)-1\neq 0,\\\\
\ln(1+t)&\mbox{otherwise.}
\end{array}\right. $$
\par For shifted process $X^{(s)}_u= L_{\tau_u}-L_{\tau_s}, u\geq s$ the corresponding moments
$$m^{(n)}_{s,t}= {\bf E}\left[\left(\int_s^te^{-(L_{\tau_u}-L_{\tau_s})}du\right)^n\,\right]=
{\bf E}\left[\left(\int_s^te^{-L_{(\tau_u -\tau_s)}}du\right)^n\,\right]$$
since $L$ is homogeneous process and 
$(L_{\tau_u}-L_{\tau_s})_{u\geq s}\stackrel{\mathcal{L}}{=}
(L_{\tau_u-\tau_s})_{u\geq s}$.
\par We change the variable putting $\tau_u-\tau_s= r\ln (1+u)-r\ln (1+s)=v-s$ where $v$ is new variable. We denote $v(s,t)= r\ln (\frac{1+t}{1+s})$ and we remark that $1+u= (1+s)\,\exp( (v-s)/r)$. Then,
$${\bf E}\left[\left(\int_s^te^{-(L_{\tau_u -\tau_s})}du\right)^n\,\right]={\bf E}\left[\left(\frac{1+s}{r}\int_s^{s+v(s,t)}e^{-\tilde{L}_{v-s}}dv\right)^n\,\right]$$
$$=\frac{(1+s)^n}{r^n}{\bf E}\left[\left(\int_0^{v(s,t)}e^{-\tilde{L}_{u}}du\right)^n\,\right]$$
Finally, we get for $n\geq 0$:
$$m^{(n)}_{s,t}= \frac{(1+s)^n}{r^n}\,\tilde{m}^{(n)}_{v(s,t)}$$
where $\tilde{m}^{(n)}_{v(s,t)}$ is n-th moment of the exponential functional of $\tilde{L}$ on $[0, v(s,t)]$.
We suppose for simplicity that $\tilde{\Phi}$ is strictly monotone on the interval $[0, n+1]$. Then,
using the integral equation of Theorem \ref{t1} and the expression for the moments of Levy processes given in Corollary \ref{c0} below, and the fact that
$$e^{- \Phi(s,n+1)} = e^{-\Phi(n+1) \tau_s}= (1+s)^{-r\Phi(n+1)} $$
we get for $n\geq 1$:
$$m^{(n+1)}_{t}= \frac{(n+1)!}{r^n}\sum_{k=0}^{n-1} 
\int_0^t (1+s)^{q(n)}\,\frac{\, e^{-\tilde{\Phi}(k)v(s,t)}-e^{-\tilde{\Phi}(n)v(s,t)}}{\displaystyle\prod_{0\leq i\leq n,\, i\neq k}\,(\tilde{\Phi}(i)-\tilde{\Phi}(k)) }\,ds
$$ 
where $q(n)= n-r\Phi(n+1)$ and $\tilde{\Phi}(k)$ is the Laplace exponent of Levy process $\tilde{L}$, $\tilde{\Phi}(k)= \Phi(k)-\frac{k}{r}$, $1\leq k\leq n$.
To express the final result we put $$\rho(k)= r\Phi(k)-k,\,\,\,
 \gamma(n,k)=n-k-r(\Phi(n+1)-\Phi(k))$$ and
$$Q_t(n,k)= \frac{(1+t)^{\gamma(n,k)+1}-1}{(\gamma(n,k)+1)(1+t)^{\rho(k)}}$$
Then, after the integration we find that  for $n\geq 1$
$${\bf E}(I^{n+1}_t)=m^{(n+1)}_{t}= (n+1)!\,
\sum_{k=0}^{n-1} \,\frac{\, Q_t(n,k)-Q_t(n,n)}{\displaystyle\prod_{0\leq i\leq n,\, i\neq k}\,(\rho(i)-\rho(k)) } $$
\end{ex}
\begin{rem}\rm It is clear that the explicit formulas for the moments in non-homogeneous case will be rather exceptional. For numerical calculus the  following formula  could be useful: for all $0\leq s\leq t$
$$m^{(\alpha)}_{s, t}=\alpha \int _s^t m^{(\alpha -1)}_{u,t}
 \,\,e^{\Phi (u,\alpha)-\Phi (s,\alpha)}\, du$$
To obtain this formula it is sufficient to find the Laplace exponent of shifted process
 $(X^{(s)}_u)_{s\leq u\leq t}$ with $X^{(s)}_u= X_u-X_s$. Since $X$ is a PII, the variables $X_s$ and $X_u - X_s$ are independent, and 
 $${\bf E}(e^{-\alpha X_u})= {\bf E}(e^{-\alpha X_s})\,{\bf E}(e^{-\alpha (X_u-X_s)})$$
 and then, the Laplace exponent of shifted process $X^{(s)}$ is given by:
 $$\Phi^{(s)}(u,\alpha )=\Phi(u, \alpha) - \Phi(s, \alpha ).$$
\end{rem} 
\end{section}
\begin{section}{Positive moments of $I_t$ and $I_{\infty}$ for Levy processes}
\par Now we suppose that $X$ is Levy process with the parameters $(b_0,c_0, K_0)$. In this case the relation  \eqref{cond} become:
\begin{equation}\label{cond1}
\int_{\mathbb{R}\setminus\{0\}} (x^2\wedge |x|) K_0(dx)  < \infty. 
\end{equation}
It should be noticed that in \cite{BY} the condition on Levy measure was
\begin{equation}\label{cond2}
\int_{\mathbb{R}\setminus\{0\}} (|x|\wedge 1) K_0(dx)  < \infty
\end{equation} 
and this condition is stronger at zero and weaker at infinity then \eqref{cond1}.
When $X$ is Levy process,  the condition \eqref{rt1} become: there exists $\delta >~0$
\begin{equation}\label{l1}
\int_{x<-1} e^{-(\alpha + \delta )  x} K_0(dx)  < \infty .
\end{equation}
\begin{corr} \label{c0} Let $n\geq 1$  and suppose that $\Phi$ is bijective on $[0,n]\cap\mathbb{N}$. Then,
$${\bf E}(I_{t}^n) = n! \sum_{k=0}^{n-1} \,\frac{\, e^{-\Phi(k)t}-e^{-\Phi(n)t}}{\displaystyle\prod_{0\leq i\leq n,\, i\neq k}\,(\Phi(i)-\Phi(k)) }$$
\end{corr}
\begin{proof}The formula follows from \eqref{8aab} by induction.
\end{proof}

\par We present here two examples, one of them is related with  Brownian motion, and second one with compound Poisson process.
\begin{ex} \rm
Let $X$ be  a Brownian motion with drift $\mu>0$ and diffusion coefficient $\sigma >0$, i.e. $X_t= \mu t +\sigma W_t$ where $W=(W_t)_{t\geq0}$ is standard Brownian motion. Then, $\Phi(\alpha) = \alpha \mu - \frac{\alpha^2\,\sigma^2}{2}$ and if $\frac{2\mu}{\sigma^2}$ is not an integer, we get:
$${\bf E}(I_{t}^n) = n! \sum_{k=0}^{n-1} \,\frac{\, e^{-(k\mu - k^2\sigma^2/2)t}-e^{-(n\mu - n^2\sigma^2/2)t}}{\displaystyle\prod_{0\leq i\leq n,\, i\neq k}\,(i-k)(\mu - (i+k)\sigma^2/2) }$$
\end{ex}
\begin{ex}\rm
Let $X$ be compound Poisson process such that $X_t= \sum_{k=1}^{N_t} U_k$ where $(U_k)_{k\geq 0}$ is a sequence of independent random variables with distribution function $F$ and $N$ is a homogeneous Poisson process with intensity $\lambda >0$. Then, $\Phi (\alpha) = \lambda\int_{\mathbb{R}\setminus\{0\}} (1-e^{-\alpha x})F(dx)$. In particular, if the $U_k$'s are standard normal variables, we get that $\Phi(\alpha)=\lambda ( 1-e^{\alpha^2\,/2})$ and
$${\bf E}(I_{t}^n) = n! \sum_{k=0}^{n-1} \,\frac{\, \exp(-\lambda t(1-e^{k^2/2}))-\exp(-\lambda t(1-e^{n^2/2}))}{\lambda^n\,\displaystyle\prod_{0\leq i\leq n,\, i\neq k}\,(e^{k^2/2}- e^{i^2/2}) }$$
\end{ex}
\par We introduce the Laplace-Carson transform $\hat{m}^{(\alpha )}_q$ of $m^{(\alpha )}_t$ of the parameter $q>~0$:
$$\hat{m}^{(\alpha )}_q= \int_0^{\infty}qe^{-qt}m^{(\alpha)}_tdt$$
This integral is always well-defined in general sense, since the integrand is positive.
\begin{corr} \label{c1}(cf. \cite{BY}) Let $X$ be a Levy process which verifies \eqref{cond1} and \eqref{l1}, and such that for fixed $\alpha\geq 1$, $m^{(\alpha)}_{\infty}<\infty$. Then the Laplace-Carson transforms $\hat{m}^{(\alpha)}_{q}$ and $\hat{m}^{(\alpha -1)}_{q}$ of $m^{(\alpha)}_{t}$ and $m^{(\alpha -1)}_{t}$ respectively, are well-defined and  we have a recurrent formula:
\begin{equation}\label{9a}
\hat{m}^{(\alpha)}_q\,(q + \Phi(\alpha ))= \alpha \hat{m}^{(\alpha -1)}_q
\end{equation}
In particular,  for integer $n\geq 1$ such that $m^{(n)}_{\infty}<\infty$ we get:
\begin{equation}\label{10a}
\hat{m}^{(n)}_q=\frac{n!}{\prod_{k=1}^n(q + \Phi(k))}
\end{equation} 
As a consequence,
$$ {\bf E}(I_{\infty}^n) = \frac{n!}{\prod_{k=1}^n \Phi(k)}$$
Moreover, if all positive moments of $I_{\infty}$ exist and the series below  converges,  then   the Laplace transform of $I_{\infty}$ of parameter $\beta\geq 0$ is given by:
\begin{equation}\label{lap} {\bf E}(e^{-\beta\, I_{\infty}}) = \sum _{n=0}^{\infty} \frac{(-1)^n\beta ^n}{\prod_{k=1}^n \Phi(k)}
\end{equation}
\end{corr}
\begin{proof} The first equality for the Laplace-Carson transforms follows directly from   \eqref{8aa}. The second equality can be obtained as particular case from the first one, by recurrence. 
\par For the third one we prove that $\hat{m}^{(n)}_q\rightarrow m^{(n)}_{\infty}$ as $q\rightarrow 0$. In fact,
  $m_t^{(n)}\rightarrow m_{\infty}^{(n )}$ as $t\rightarrow +\infty $ and
 $$\lim_{t\rightarrow +\infty}\frac{1}{t}\int _0^t\,m_s^{(n )} ds = m_{\infty }^{(n)}.$$ 
Let us denote $M^{(n)}_t=\displaystyle\int _0^t\,m_s^{(n )} ds$. By integration by part formula  we have for each $t>~0$ :
$$\int_0^{t} qe^{-qs} m_s^{(n )} ds= [q\,e^{-qt}M^{(n)}_t]_0^t + \int_0^{t} q^2e^{-qs}M_s^{(n )} ds$$
Then, since $M_0^{(n)}=0$ and $M^{(n)}_t/t \rightarrow m_{\infty}^{(n )}$ as $t\rightarrow +\infty$,
$$\hat{m}^{(n)}_q=\int_0^{\infty} qe^{-qs} m_s^{(n)} ds =  q^2 \int_0^{\infty}e^{-qs}M^{(n)}_s ds $$
Since $q^2\,\int_0^{\infty}s\, e^{-qs} ds=1$, we get
$$\hat{m}^{(n)}_q-m_{\infty}^{(n)} = q^2\,\int_0^{\infty} e^{-qs}(M^{(n)}_s-s\,m_{\infty}^{(n)}) ds$$
For each $\epsilon >0$ there exists $t_{\epsilon}$ such that for $s\geq t_{\epsilon}$, $|\frac{M^{(n)}_s}{s}-m_{\infty}^{(n)}|\leq \epsilon$.
Then, 
$$\hspace{-3cm}|\,\hat{m}^{(n)}_q-m_{\infty}^{(n)}\,|\leq  q^2\,\int_0^{t_{\epsilon}} e^{-qs}|M^{(n )}_s-s\,m_{\infty}^{(n)}| ds+$$
$$\hspace{1cm}q^2\,\int_{t_{\epsilon}}^{\infty}s\, e^{-qs}\left|\frac{M^{(n)}_s}{s}-m_{\infty}^{(n)}\right| ds\leq 
q^2\,\int_0^{t_{\epsilon}} e^{-qs}|M^{(n)}_s- s\,m_{\infty}^{(n)}| ds + \epsilon$$
We notice that
$$\lim_{q\rightarrow 0}q^2\int_0^{t_{\epsilon}} e^{-qs}\,|\,M^{(n)}_s-s\,m_{\infty}^{(n )}\,|\, ds= 0$$
Then, taking $\lim_{\epsilon\rightarrow 0}\lim_{q\rightarrow 0}$ in the previous inequality we get  that $$\hat{m}^{(n)}_q\rightarrow~ m^{(n)}_{\infty}$$ as $q\rightarrow 0$. Finally, we take the limit as $q\rightarrow 0$ in second equality, to obtain the third one. The formula for the Laplace transform of $I_{\infty}$ can be proved by using Taylor expansion with remainder in Lagrange form.
\end{proof}

\begin{ex}\rm Let  $X$ be homogeneous Poisson process of intensity $\lambda >0$. Then all positive moments of $I_{\infty}$ exist, and we have for $0\leq \beta <\lambda$
$$ {\bf E}(e^{-\beta\, I_{\infty}}) = \sum _{n=0}^{\infty} \frac{(-1)^n\beta ^n}{\lambda^n\,\prod_{k=1}^n (1-e^{-k})}.$$

\end{ex}
\begin{corr} \label{c2} Let 
$\alpha_0 = \inf\{\alpha >0\,|\, \Phi (\alpha)\leq 0\}$ with $\inf\{\emptyset\}=+\infty$. Then, ${\bf E}(I^{n}_{\infty})<~\infty$ if and only if $1\leq n<\alpha_0$. In particular,  for Brownian motion with drift coefficient $b_0$ and diffusion coefficient  $c_0\neq~0$, 
$$\Phi (n)= nb_0 - \frac{1}{2}n^2 c_0$$ and the moment of $I_{\infty}$ of order $n\geq 1$ will exist if $n< \frac{2b_0}{c_0}$. \\
If $X$ is a subordinator with non-zero Levy measure $K_0$ such that \eqref{cond2} holds, then 
$$\Phi (n)=n[b_0 - \int_{\mathbb{R}^+\setminus\{0\}}xK_0(dx)\,] - \int_{\mathbb{R}^+\setminus\{0\}}(e^{-nx}-1)K_0(dx),$$
and under the condition
  $$b_0 - \int_{\mathbb{R}^+\setminus\{0\}}xK_0(dx)\,\geq 0,$$  all moments of $I_{\infty}$ exist.
\end{corr}
\begin{proof} Let $n=\sup\{ k\geq1 : {\bf E}(I^k_{\infty})<\infty\}$. If $n=+\infty$, then $\Phi(k)>0$ for all $k\geq 1$ and $\alpha_0=+\infty$. If $1\leq n<+\infty$,
from  Corollary \ref{c1} we get that $\Phi(n)>0$.
Since $\Phi(\alpha)$ is concave function such that $\Phi (0)=0$, we conclude that $n< \alpha_0$.
Conversely, substituting $t-s$ by $s$ in \eqref{8aab} of Theorem \ref{t1} we get:
$${\bf E}(I^n_{t}) \leq n \,{\bf E}(I^{n-1}_{\infty})\int _0^t e^{-\Phi(n)s}\,ds$$
If $1\leq n<\alpha_0$,  $\Phi (n) >0$, and the integral on the r.h.s. of this inequality  converge as $t\rightarrow \infty$. By induction, it gives that ${\bf E}(I^{n}_{\infty})<~\infty$. Moreover, the results for continuous case and the case when $X$ is a subordinator, follow immediately  from the expression of $\Phi (k)$.
\end{proof}
\begin{ex}\rm Let $X$ be time changed Brownian motion, namely $X_t= \mu \tau_t +\sigma W_{\tau_t}$ where $W=(W_t)_{t\geq 0}$ is standard Brownian motion, $\mu\in\mathbb{R}
$, $\sigma>0$ and $\tau_t$ is first hitting time of the level $t$ of the independent (from $W$) standard Brownian motion $B=(B_t)_{t\geq 0}$ with the drift coefficient $b>0$. Then, as known,
$ \Phi(\alpha)= (b^2+2\alpha \mu-\alpha^2\sigma^2)^{1/2}-b$ with $b^2+2\alpha \mu-\alpha^2\sigma^2>0$ (see for instance \cite{BS}, formula 2.0.1, p.295). Then, ${\bf E}(I^n_{\infty})<\infty$ if and only if $2\mu-n\sigma >0$.
 
\end{ex}
\begin{ex}\rm Let $X$ be pure discontinuous Levy process with Levy measure
$$K_0(dx) = \frac{c\exp(-Mx)}{x^{1+\beta}}{\bf I}_{]0,+\infty[}(x)dx$$
where $c>0, M>0, 0<\beta<1 $. Then, $$\Phi(\alpha) = \frac{c\,\Gamma(1-\beta )}{-\beta} ((M+\alpha)^{\beta}-M^{\beta} - \alpha M^{\beta-1}\beta).$$ Then, $\Phi(\alpha) >0$ for $\alpha \geq 1$, and all moments of $I_{\infty}$ exist.
\end{ex}
\end{section}
\begin{section}{Recurrent formulas for the Mellin transform  of $I_t$ 
with $\alpha <0$.}
\par In the following Theorem \ref{t2}, we derive the  integro-differential equation for the Mellin transform  $m_t^{(\alpha )}$  of $I_t$ with $\alpha<0$.
\begin{thm}\label{t2} Let $\alpha< 0$  and $t>0$ be fixed. Suppose that \eqref{abscont} and\eqref{cond} hold and 
\begin{equation}\label{rt11}
\int_0^t\int_{x>1} e^{(|\alpha |+1)  x} K_s(dx)\, ds < \infty .
\end{equation}
Then, for $s>0$,  $m^{(\alpha)}_{s,t}$ is well-defined as well as $m^{(\alpha -1)}_{s,t}$ and we get
the following recurrent  differential equation:
\begin{equation}\label{8aaa}
m^{(\alpha -1)}_{s,t} =  \frac{1}{\alpha}\left(  m^{(\alpha )}_{s,t}\,H^{(\alpha)}_s-\frac{d}{ds} m^{(\alpha )}_{s,t}\right)
\end{equation}
In the case of Levy process $X$ we have:
\begin{equation}\label{levy2}
m^{(\alpha -1)}_{s} =  \frac{1}{\alpha} \left( m^{(\alpha )}_{s}\,\Phi (\alpha)+\frac{d}{ds} m^{(\alpha )}_{s}\right)
\end{equation}
\end{thm}
\begin{proof}The proof of this result is similar to the proof of Theorem \ref{t1}.
For $n\geq 1$ we introduce the stopping times 
$$\tau_n= \inf \{u\geq  s: V_u\leq \frac{1}{n}\,\mbox{or}\, \exp(-Y_u)\geq n\}$$
with $\inf\{\emptyset\}=+\infty$. Then from the Ito formula similarly to \eqref{v1} we get: for $0<s<t$
\begin{equation}\label{v12}
V^{\alpha}_{t\wedge \tau_n} = V^{\alpha}_{s} + \alpha \int_s^{t\wedge \tau_n} V_{u-}^{\alpha -1} dV_u + \frac{1}{2}\alpha (\alpha -1)\int_s^{t\wedge \tau_n} V_{u-}^{\alpha -2}d<V^c>_u$$
$$ + \int_s^{t\wedge \tau_n}\int_{\mathbb{R}\setminus\{0\}}\left( (V_{u-}+x)^{\alpha} - V^{\alpha}_{u-} - \alpha V^{\alpha -1}_{u-} x\right) \mu_V(du,dx)
\end{equation}
where $\mu_V$ is the measure of the jumps of $V$.
Using  \eqref{v2}, \eqref{v3}, \eqref{v4} we have
\begin{equation}\label{92}
V^{\alpha}_{t\wedge \tau_n} =V^{\alpha}_{s}+ \alpha \int_s^{t\wedge \tau_n} V_{u-}^{\alpha -1} du 
\end{equation} 
$$+ \alpha \int_s^{t\wedge \tau_n}J_{u}\, V_{u-}^{\alpha -1} d(e^{-Y_u}) + \frac{1}{2}\alpha (\alpha -1)\int_s^{t\wedge \tau_n} V_{u-}^{\alpha }d<Y^c>_u$$
$$\hspace{3cm} + \int_s^{t\wedge \tau_n}\int_{\mathbb{R}\setminus\{0\}}V_{u-}^{\alpha}\left( e^{-\alpha x} -1-\alpha(e^{-x}-1)\right)\mu_Y(du,dx)$$
where $\mu_Y$  the measure of the jumps of $Y$. Taking in account \eqref{v5} we, finally, find that
\begin{equation}\label{rel12}
{\bf E}(V^{\alpha}_{t\wedge \tau_n}) = {\bf E}(V^{\alpha}_{s})+\alpha {\bf E}\left(\int_s^{t\wedge \tau_n} V_{u-}^{\alpha -1} du\right) - {\bf E} \left(\int_s^{t\wedge \tau_n}V_{u-}^{\alpha}\,\,d\bar{\Phi}(u, \alpha) \right)
\end{equation}
\par We remark that $\tau_n\rightarrow +\infty\, ({\bf P}-a.s.)$ as $n\rightarrow +\infty$. To pass to the limit as $n\rightarrow \infty$ in the r.h.s. of the above equality, we use   the Lebesgue monotone convergence theorem for the first term and the Lebesgue dominated convergence theorem for the second term. In fact,   for second term we have:
$$ \left|\int_s^{t\wedge \tau_n}V_{u-}^{\alpha}\,\,d\bar{\Phi}(u, \alpha)\right|\leq \int_s^{t}V^{\alpha}_{u-}\,|\bar{H}_u^{(\alpha )}|du$$
In addition,
$${ \bf E}\left(\int_s^{t}V^{\alpha}_{u-}\,|\bar{H}_u^{(\alpha )}|du\right)\leq 
\sup_{s\leq u \leq t}{\bf E}(V^{\alpha}_u)\,\int_s^{t}|\bar{H}_u^{(\alpha )}|du$$
The function $(\bar{H}_u^{(\alpha )})_ {0\leq u\leq t}$ is deterministic function, integrable on finite intervals. Hence, it remains to show that
\begin{equation}\label{sup}
\sup_{s\leq u \leq t}{\bf E}(V^{\alpha}_u) < \infty .
\end{equation} 
By the Jensen inequality
$$V^{\alpha}_u\leq u^{\alpha -1}\int_0^u e^{\alpha (Y_v - Y_u)}dv$$
Then, in the same way as in Theorem \ref{t1} we get that
$$\sup_{s\leq u \leq t}{\bf E}(V^{\alpha}_u)\leq s^{\alpha }\exp\{ \int_{0}^{t}| H^{(\alpha)}_r|dr \} $$
\par To pass to the limit in the l.h.s. of \eqref{rel12}, we  prove that  the family of $(V_{u\wedge\tau_n}^{\alpha})_{n\geq 1}$ is uniformly integrable, uniformly in  $u\in[s,t]$ in the same way as in Theorem \ref{t1}. For that we introduce a submartingale $(\bar{V}^{(\alpha)}_u)_{s\leq u\leq t}$ with $$\bar{V}^{(\alpha)}_u = e^{-\alpha\bar{M}_u}\int _0^u e^{\alpha\bar{M}_v} dv$$
We remark that 
$$V_u^{\alpha}\leq e^{|\alpha|\,Var(\bar{B})_t} \,\, \bar{V}^{(\alpha)}_u.$$
In addition, on the set $\{\bar{V}^{(\alpha)}_{u\wedge\tau _n}>c\}$ we have: $1< (\frac{\bar{V}^{(\alpha)}_{u\wedge\tau _n}}{c})^{\frac{1}{|\alpha |}}$. Then
$$ \bar{V}^{(\alpha)}_{u\wedge\tau _n} {\bf I}_{\{\bar{V}^{(\alpha)}_{u\wedge\tau _n}>c\}}\leq 
 (\bar{V}^{(\alpha)}_{u\wedge\tau _n})^{1+\frac{1}{|\alpha |}}\,c^{-\frac{1}{|\alpha |}}=
 (\bar{V}^{(\alpha)}_{u\wedge\tau _n})^{\frac{\alpha -1}{\alpha}}\,c^{-\frac{1}{|\alpha |}} $$
Applying Jensen inequality we get that
$$(\bar{V}^{(\alpha)}_{u\wedge\tau _n})^{\frac{\alpha -1}{\alpha}}\leq \bar{V}^{(\alpha-1)}_{u\wedge\tau _n}\,u^{\frac{1}{|\alpha |}}$$
Hence, we proved  that for all $n\geq 1$, $c>0$ and $s\leq u\leq t$
$${\bf E}( \bar{V}^{(\alpha)}_{u\wedge\tau _n} {\bf I}_{\{\bar{V}^{(\alpha)}_{u\wedge\tau _n}>c\}}) \leq 
 c^{-\frac{1}{|\alpha|}}\, \max (s^{\frac{1}{|\alpha|}}, t^{\frac{1}{|\alpha|}})\,{\bf E}(\bar{V}_u^{(\alpha -1)})$$
Finally, we prove that
$$\sup_{s\leq u \leq t}{\bf E}(\bar{V}^{(\alpha-1)}_u)<\infty $$
in the same manner as before, and it implies the uniform integrability.
 \par After limit passage we get that
\begin{equation}\label{rt2}
{\bf E}(V_t^{\alpha}) = {\bf E}(V_s^{\alpha}) - \int_s^{t}{\bf E}(V_u^{\alpha}) \,d\bar{\Phi}(u,\alpha ) + \alpha\int_s^t
{\bf E}(V_u^{\alpha -1})du
\end{equation}
We differentiate  each term of this equality w.r.t. $s$ to obtain that
$$\frac{d}{ds} {\bf E}(V_s^{\alpha}) + {\bf E}(V_s^{\alpha})\,\bar{H}^{(\alpha) }_s - \alpha
{\bf E}(V_s^{\alpha -1})=0 $$
We take in account that $ {\bf E}(V_s^{\alpha})= m^{(\alpha)}_{t-s,t}$,
$ {\bf E}(V_s^{\alpha-1})= m^{(\alpha-1)}_{t-s,t}$ and that $\bar{H}^{(\alpha) }_s=H^{(\alpha) }_{t-s}$. This implies that
$$\frac{d}{ds}m^{(\alpha)}_{t-s,t}+m^{(\alpha)}_{t-s,t}H^{(\alpha )}_{t-s}-\alpha\,m^{(\alpha-1)}_{t-s,t}=0$$
Finally, replacing $t-s$ by $s$ we get \eqref{8aaa}. In the case of Levy processes
$m^{(\alpha)}_{t-s,t}=m^{(\alpha)}_{s}$ due to homogeneity, and $H^{(\alpha )}_{t-s}=\Phi(\alpha)$, and this gives \eqref{levy2}.
\end{proof}
\par To present the results about negative moments of Levy process $X$  with the parameters $(b_0,c_0, K_0)$, we introduce  the condition:
\begin{equation}\label{l22} 
\int_{x>1} e^{(|\alpha | +1) x} K_0(dx) < \infty .
\end{equation}
\begin{corr} \label{c3}(cf. \cite{BY})Let $X$ be Levy process verifying \eqref{cond1} and \eqref{l22}, and let $\alpha\leq~-1$ be fixed. Suppose that $m^{(\alpha -1 )}_{\infty}<\infty$. Then the Laplace-Carson transforms  of $m^{(\alpha)}_{\infty}$ and $m^{(\alpha -1)}_{\infty}$ are well-defined and  we have a recurrent formula:
\begin{equation}\label{9aa}
\hat{m}^{(\alpha -1)}_q=\frac{1}{\alpha}\left(\hat{m}^{(\alpha )}_q \Phi(\alpha)+q\,\hat{m}^{(\alpha )}_q \right)
\end{equation}
and, hence,
$$\hat{m}^{(\alpha -1)}_q= \frac{1}{\alpha}\hat{m}^{(\alpha )}_q (q+\Phi(\alpha))$$
In particular, under above conditions,  for integer $n\geq 2$ and $\alpha = -n$ we get:
\begin{equation}\label{10aa}
\hat{m}^{(-n)}_q=  \hat{m}^{(-1)}_q\,\cdot\frac{(-1)^{n-1}}{(n-1)!}\,\prod_{k=1}^{n-1}(q+\Phi(-k) )
\end{equation}
As a consequence,
\begin{equation}\label{neg}
 {\bf E}(I_{\infty}^{-n}) = {\bf E}(I_{\infty}^{-1})\,\cdot\frac{(-1)^{n-1}}{(n-1)!}\,\,\prod_{k=1}^{n-1} \Phi(-k)
\end{equation}
\end{corr}
\begin{proof}Two first equalities given above follow directly from  Theorem \ref{t2}. The third one can be proved in the same way as in Corollary \ref{c1} by letting $q\rightarrow 0$.
\end{proof}
\begin{ex}\rm
For compound Poisson process presented in Example 6 with $U_k$'s which follows standard normal distribution, we get for $n\geq 1$
$$  {\bf E}(I_{\infty}^{-n}) = {\bf E}(\,I_{\infty}^{-1})\,\frac{\lambda^{n-1} }{(n-1)!}\,\,\prod_{k=1}^{n-1}(e^{\frac{k^2}{2}}-1).$$
\end{ex}
\begin{corr} \label{c4} Let $\beta = \sup\{ k\geq 1\,|\, -\infty < \Phi(-l)<0 \,\,\mbox{for} \,\,1\leq l\leq k\}$
with $\sup\{\emptyset\}=1$. Then ${\bf E}(I^{-(n+1)}_{\infty})<~\infty$ if and only if $n\leq \beta$ and  ${\bf E}(I^{-1}_{\infty})<~\infty$.\\
In particular, for Brownian motion with the drift coefficient $b_0$ and the diffusion coefficient $c_0\neq 0$, ${\bf E}(I^{-(n+1)}_{\infty})<~\infty$ if and only if $\frac{2b_0}{c_0}>-1$.\\ 
If $X$ is a subordinator with $\int_{\mathbb{R}^+\setminus\{0\}}xK_0(dx) < \infty$, then 
$$\Phi (-k)=- k[b_0 - \int_{\mathbb{R}^+\setminus\{0\}}xK_0(dx)] - \int_{\mathbb{R}^+\setminus\{0\}}(e^{kx}-1)K_0(dx),$$ and
under the condition 
\begin{equation}\label{condition}
b_0 - \int_{\mathbb{R}^+\setminus\{0\}}xK_0(dx)\geq 0
\end{equation}
${\bf E}(I^{-(n+1)}_{\infty})<~\infty$  if and only if  $\int_{\mathbb{R}^+\setminus\{0\}}(e^{nx}-1)K_0(dx)<\infty$.
\end{corr}
\begin{proof} Suppose that ${\bf E}(I^{-(n+1)}_{\infty})<~\infty$ for some $n>0$. Then by Cauchy-Schwartz inequality we get that for all $k$, $1\leq k\leq n$, ${\bf E}(I^{-k}_{\infty})<~\infty$. Then the formula \eqref{neg} yields that  ${\bf E}(I^{-1}_{\infty})<~\infty$ and
$-\infty<\Phi(-k)<0$ for $1\leq k\leq n$. Hence, $n\leq \beta$ and  ${\bf E}(I^{-1}_{\infty})<~\infty$.\\
Conversely, if $n\leq \beta$ and  ${\bf E}(I^{-1}_{\infty})<~\infty$, then $-\infty<\Phi(-k)<0$ for $1\leq k\leq n$. Then from \eqref{levy2} we deduce that
$$m_s^{-(k+1)}= \frac{\Phi(-k)}{-k}m^{(-k)}_s -\frac{1}{k} \frac{d}{ds}m^{(-k)}_s\leq  \frac{\Phi(-k)}{-k}m^{(-k)}_s \leq  \frac{|\Phi(-k)|}{k}m^{(k)}_{\infty}$$
since $\frac{d}{ds}m^{(-k)}_s\geq 0$. Hence,
$${\bf E}(I^{-(n+1)}_{\infty})= m_{\infty}^{-(n+1)}\leq \prod_{k=1}^n \frac{|\Phi(-k)|}{k}\,{\bf E}(I^{-1}_{\infty})<\infty .$$
In the case of Brownian motion we conclude that $\Phi(-k)= -kb_0-\frac{1}{2}k^2 c_0 >-\infty$ for all $k\geq 1$. Since $\Phi$ is concave function with $\Phi(0)=0$, the condition $\Phi(-1)<0$ ensures the existence of all negative moments.\\
In the case when $X$ is a subordinator, and under mentioned condition \eqref{condition},  all $\Phi(-k)<0$ and only the condition of finiteness of $\Phi(-k)$ is involved in the existence of the negative moments.
\end{proof}
\begin{ex}\rm Let us apply the Corollary \ref{c4} to 
 time changed Brownian motion considered in Example 8. We get that $\Phi(\alpha)<~0$
whenever $-b^2\leq 2\alpha\mu - \alpha ^2\sigma^2<0$. Hence, all negative moments of $I_{\infty}$ exists if ${\bf E}(I_{\infty}^{-1})<\infty$ and $2\mu +\sigma^2>0$. 
\end{ex}
{\bf Acknowledgement}\\ We are grateful to our referees for very useful remarks and comments.

\end{section}

\end{document}